\documentclass[12pt]{article}
\usepackage{amsmath}
\usepackage{amsfonts}
\usepackage{amssymb}
\usepackage{graphicx}
\usepackage{latexsym}

\newcommand{\A} {\mathcal{A}}
\newcommand{\FF} {\mathcal{FF}}
\renewcommand{\AA} {\mathcal{AA}}
\newcommand{\B} {\mathcal{B}}
\newcommand{\BO} {\mathcal{BO}}

\newtheorem{theorem}{Theorem}

\newtheorem{lemma}[theorem]{Lemma}

\newenvironment{proof}{\noindent\emph{Proof}.\ }{\eop}
\newcommand\Bbox{
{\unskip\nobreak\hfil\penalty50
\hskip1em\hbox{}\nobreak\hfil{\lower .5pt \hbox{$\Box$}}
\parfillskip=0pt \finalhyphendemerits=0 \par}
}
\newcommand\eop{\ifmmode {\hbox{\Bbox}} \else \Bbox \fi}

\begin{document}

\title {{\bf Online coloring graphs with high girth and high oddgirth}}

\author{{J. Nagy-Gy\"orgy}\thanks {Department of Mathematics, University of
Szeged, Aradi V\'ertan\'uk tere 1, H-6720 Szeged, Hungary, email:
Nagy-Gyorgy@math.u-szeged.hu}}

\date{}

\maketitle

\begin{abstract}
We give an upper bound for the online chromatic number of graphs with high girth and for graphs with high oddgirth generalizing Kierstead's algorithm for graphs that contain neither a $C_3$ or $C_5$ as an induced subgraph.
\end{abstract}

{\it keywords:} online algorithms, combinatorial problems

\section{Introduction}
A considerable effort has been devoted to color graphs with high girth or high oddgirth. At the same time online coloring has attracted major attention. We unify these two lines of research and investigate how online coloring techniques can be applied to color graphs with high girth or oddgirth.

A coloring of a graph is an assignment of positive integers to the vertices of the
graph so that every edge contains vertices having different colors.
A $k$-coloring of a graph is a coloring of it where the number of used colors is at most $k$.
In the online graph coloring problem the algorithm receives the vertices of
the graph in some order $v_1,\dots,v_n$ and it must color $v_i$ by only looking at
the subgraph induced by $V_i=\{v_1,\dots,v_i\}$.

Online graph coloring has been investigated in several papers, one can find many details
on the problem in the survey paper \cite{KIE1}. 
Some results are proved about the straightforward online graph coloring algorithm
First Fit (which we denote by $\FF$). $\FF$ uses the smallest color for each vertex which does not
make a monochromatic edge. In \cite{GYA} it is shown that this algorithm is the best possible for the trees.
In \cite{LOV}  an online algorithm is presented which
colors $k$-colorable graphs on $n$ vertices with at most $O(n\log^{(2k-3)} n/ \log^{(2k-4)} n)$ colors.
The best known lower bound \cite{Vis} states that no online algorithm can color every $k$-colorable
graph on $n$ vertices with less then $\Omega(\log^{k-1} n)$ colors. 
In \cite{KIE2} an online algorithm is presented which colors $k$-colorable graphs on $n$ vertices with at most $O(n^{1-1/k!})$ colors. The author of \cite{KIE2} introduces an algorithm $\B_n$ which uses less than $3n^{1/2}$ colors to color any online graph on $n$ vertices that induces neither $C_3$ nor $C_5$. In this paper we will generalize this result for graphs with high girth and graphs with high oddgirth. In the latter case we will apply an online algorithm developed by Lov\'asz (see section \ref{oddg}) which can color a bipartite graph on $n$ vertices with at most $2 \log_2 n$ colors.

\section{Preliminaries}

For a graph $G$ the minimum number of colors which is enough to color the
graph is called the chromatic number of the graph and denoted by $\chi(G)$.

An online graph is a structure $G^{\prec}=(G,{\prec})$ where $G=(V,E)$ is a graph and ${\prec}$ is a linear order of its vertices. Let $G_i$ denote the online graph induced by the ${\prec}$-first $i$ elements $V_i$ of $V$.

An online graph coloring algorithm colors the $i$-th vertex of the graph  by only looking at
the subgraph $G_i$.
For an online algorithm $\A$ and an online graph $G^{\prec}$, the number of
colors used by $\A$ to color $G^{\prec}$ is denoted by $\chi_\A(G^{\prec})$.
For a graph $G$, $\chi_\A(G)$ denotes the maximum of the $\chi_\A(G^{\prec})$
values over all orderings ${\prec}$. 

The girth of a graph $G$ denoted $g(G)$ is the length of its shortest cycle and the oddgirth of $G$ denoted $g_o(G)$ is length of its shortest odd cycle.

The distance  $\mathrm{dist}(u,v)$ of vertices $u$ and $v$ is the length of the shortest $uv$ path. For a positive integer $d$, let $N_d(v)$ to be the set of vertices with positive distance at most $d$ from vertex $v$:
$$N_d(v) = \{ u\in V(G) : 1\le \mathrm{dist}(u,v)\le d \}.$$
$N_{d,\mathrm{odd}}(v)$ is the set of vertices with an positive \emph{odd} distance at most $d$ from vertex $v$:
$$N_{d,\mathrm{odd}}(v) = \{ u\in V(G) : 1\le \mathrm{dist}(u,v)\le d \textrm{ and } \mathrm{dist}(u,v) \textrm{ is odd}\}.$$
For any $S\subset V(G)$ let us define 
$$N_d(S)=\bigcup_{v\in S} N_d(v) \setminus S \quad\textrm{and}\quad N_{d,\mathrm{odd}}(S)=\bigcup_{v\in S} N_{d,\mathrm{odd}}(v) \setminus S.$$
Let $N^{\prec}_d(v)$ the set of vertices preceding $v$ with a positive distance at most $d$ from vertex $v$:
$$N^{\prec}_d(v) = \{ u\in V(G^{\prec}) :  u{\prec}v,\ 1\le \mathrm{dist}(u,v)\le d \}.$$
$N^{\prec}_{d,\mathrm{odd}}(v)$ is the set of vertices preceding $v$ with a positive odd distance at most $d$ from vertex $v$:
$$N^{\prec}_{d,\mathrm{odd}}(v) = \{  u\in V(G^{\prec}) :  u{\prec}v,\ 1\le \mathrm{dist}(u,v)\le d \textrm{ and } \mathrm{dist}(u,v) \textrm{ is odd}\}.$$
For any $S\subset V(G^{\prec})$ let 
$$N^{\prec}_d(S)=\bigcup_{v\in S} N^{\prec}_d(v) \setminus S \quad\textrm{and}\quad N^{\prec}_{d,\mathrm{odd}}(S)=\bigcup_{v\in S} N^{\prec}_{d,\mathrm{odd}}(v) \setminus S.$$
Note that $N_1(v)=N(v)$ and $N^{\prec}_1(v)=N^{\prec}(v)$ are just the neighbors and preceding neighbors of $v$. Furthermore $N_1(S)=N(S)$ and $N^{\prec}_1(S)=N^{\prec}(S)$.

For any $m>0$ let $[m]:=\{1,2,\ldots,m\}$.

We can assume that our algorithms know the number of vertices of $G$ by the following Lemma \cite{KIE1}.
\begin{lemma}[Kierstead \cite{KIE1}] 
Let $\Gamma$ be a class of graphs and $f$ be an integer valued function on the positive integers such that $f(x) \le f(x+1) \le f(x)+1$, for all $x$. If for every $n$, there exists an online coloring algorithm $\A_n$ such that for every graph $G\in \Gamma$ on $n$ vertices, $\chi_{\A_n}(G)\le f(n)$ then there exists a fixed online coloring algorithm $\A$ such that for every $G\in \Gamma$ on $n$ vertices $\chi_{\A}(G)\le 4f(n)$.
\end{lemma}
Kierstead's algorithm \cite{KIE1,KIE2} which will be modified is the following.
\begin{quote}
\textbf{Algorithm $\B_n$} \\
Consider the input sequence $v_1{\prec}\ldots{\prec}v_n$ of an online graph $G^{\prec}$ containing neither $C_3$ nor $C_5$. Initialize by setting $W_i=\emptyset$ for all $n^{1/2} < i < 2n^{1/2}$. At the $s$-th stage the algorithm processes the vertex $v_s$ as follows.
\begin{enumerate}
\item If there exists $i\in [n^{1/2}]$ such that $v_s$ is not adjacent to any vertex colored $i$ then color $v_s$ by the least such $i$. (Coloring by $\FF$ with $n^{1/2}$ colors.)
\item Otherwise, if there exists $i>n^{1/2}$ such that $v_s\in N(W_i)$ then color $v_s$ by the least such $i$.
\item Otherwise, let $j>n^{1/2}$ be the least integer with $W_j=\emptyset$.
Set $W_j=\{v \in N^{\prec}(v_s)$ : the color of $v$ is at most $n^{1/2}\}$ and color $v_s$ such that $j$.
\end{enumerate}
\end{quote}
The following lemma holds for algorithm $\B_n$.
\begin{lemma}[Kierstead \cite{KIE1,KIE2}]\label{kier}
Algorithm $\B_n$ produces a coloring of any graph $G^{\prec}$ on $n$ vertices that induces neither $C_3$ nor $C_5$ with fewer than $2n^{1/2}$ colors.
\end{lemma}


\section{Main results}
\subsection{Graphs with high girth}
Now we modify Kierstead's algorithm for graphs with high girth.
\begin{quote}
\textbf{Algorithm $\B_{n,d}$} \\
Let $d$ be a positive integer.
Consider the input sequence $v_1{\prec}\ldots{\prec}v_n$ of an online graph $G^{\prec}$ with $g(G^{\prec})\ge 4d+1$. Initialize by setting $W_i=\emptyset$ for all $dn^{1/(d+1)} < i < (d+1)n^{1/(d+1)}$. At the $s$-th stage the algorithm processes the vertex $v_s$ as follows.
\begin{enumerate}
\item If there exists $i\in [dn^{1/(d+1)}]$ such that $v_s$ is not adjacent to any vertex colored $i$ then color $v_s$ by the least such $i$.
\item Otherwise, if there exists $i>dn^{1/(d+1)}$ such that $v_s\in N_d(W_i)$ then color $v_s$ by the least such $i$.
\item Otherwise, let $j>dn^{1/(d+1)}$ be the least integer such that $W_i=\emptyset$. Set $W_j=N_d^{\prec}(v_s)$ and color $v_s$ with $j$. 
\end{enumerate}
\end{quote}

\begin{theorem}\label{t1}
Let $d\ge 1$ be an integer. Algorithm $\B_{n,d}$ produces a coloring of any graph $G^{\prec}$ on $n$ vertices with girth $g\ge 4d+1$ with less than $(d+1)n^{1/(d+1)}$ colors.
\end{theorem}
We note that $\B_{n,1}$ is exactly $\B_n$ so in case $d=1$ Lemma \ref{kier} holds. Our proof goes along the same line as Kierstead's proof.

\begin{proof}
First we prove that $\B_{n,d}$ produces a coloring.
Assume to the contrary that two adjacent vertices $x{\prec}y$ have the same color $j$. Clearly $y$ is not colored by Step~1. Thus $j>dn^{1/(d+1)}$ and hence $x$ is not colored by Step~1. 
Since only the first vertex colored $j$ can be colored by Step~3, $y$ must be colored by Step~2 so $y\in N_d(W_j)$. If $x$ is colored by Step~3 then $W_j\subset N_d(x)$ and $y\in N_d(W_j)$ so there exists $z\in W_j$ that both $\mathrm{dist}(x,z)\le d$ and $\mathrm{dist}(y,z)\le d$ hold. But then $G^{\prec}$ contains a cycle of length at most $2d+1$, a contradiction. 
If $x$ is colored by Step~2 then both $x$ and $y$ are in $N_d(W_j)$ so there exist (not necessarily distinct) $x',y'\in W_j$ with  $\mathrm{dist}(x',x)\le d$, $\mathrm{dist}(y',y)\le d$, $\mathrm{dist}(x',z)\le d$ and $\mathrm{dist}(y',z)\le d$ where $z$ is the first vertex colored with $j$. In this case $G^{\prec}$ contains a cycle with length at most $4d+1$, a contradiction. So $\B_{n,d}$ produces a coloring.

Now we give an upper bound for the number of colors used by $\B_{n,d}$. At most $dn^{1/(d+1)}$ colors are used in Step~1.
Let $j>dn^{1/(d+1)}$ and $z_j$ be the first vertex colored $j$. 
From the assumption on $g(G^{\prec})$ it follows that the subgraph induced by $W_j\cup\{z_j\}$ is a tree.
Since $z_j$ is not colored by Step~1, due to the greediness of Step~1, it has neighbors colored $1,2,\ldots,dn^{1/(d+1)}$ in $W_j$ and each vertex $x\in N_{d-1}(z_j)$ colored $i\le dn^{1/(d+1)}$ has neighbors colored $1,2,\ldots,i-1$ in $W_j$. Thus, for each $S\subset [dn^{1/(d+1)}]$, $|S|\le d$ there exists $x\in W_j$ such that the colors occuring on the (unique) $z_j$--$x$ path are exactly the elements of $S\cup\{j\}$. So counting the $z_j$--$x$ paths with length at most $d$ over all possible $x$ we get that 
$$|W_j|\ge \sum_{\ell=1}^{d}\binom{dn^{1/(d+1)}}{\ell} > \binom{dn^{1/(d+1)}}{d} > n^{d/(d+1)}.$$
Since $z_j\not\in N(W_i)$ if $i\ne j$ that is $W_i\cap W_j = \emptyset$, at most $n/(n^{d/(d+1)}) = n^{1/(d+1)}$ colors are used in Steps~2 and 3. 
Thus $$\chi_{B_{n,d}}(G^{\prec})\le (d+1)n^{{1}/{(d+1)}}.$$
\end{proof}

If $G^{\prec}$ is an online graph on $n$ vertices with $g(G^{\prec})=g=4d+1$ then $\chi_{B_{n,d}} \le \frac{g+3}{4}n^{{4}/{(g+3)}}$.
We note that Erd\H{o}s \cite{ERD} proved that for any $g>0$ and sufficiently large $n$ there exists a graph $G$ on $n$ vertices with girth greater than $g$ and with $\chi(G) > n^{1/(2g)}$.

\subsection{Graphs with high oddgirth}\label{oddg}
Hereinafter we will need the following online algorithm developed by Lov\'asz (posed as ''an easy exercise'' in \cite{LOV}, see \cite{KIE1} for details) as an auxiliary algorithm which colors any bipartite graph on $n$ vertices with $2 \log_2 n$ colors.
\begin{quote}
\textbf{Algorithm $\AA$} \\
Consider the input sequence $v_1{\prec}\ldots{\prec}v_n$ of a 2-colorable online graph $G^{\prec}$.\\
When $v_i$ is presened there is a unique partition $(I_1,I_2)$ of the connected component of $G_i^{\prec}$ to which $v_i$ belongs, into independent sets such that $v_i\in I_1$. Assign $v_i$ the least color not already assigned to some vertex of $I_2$.
\end{quote}
Obviously algorithm $\AA$ produces a coloring.
The following lemma is a slight improvement of Theorem 2 of \cite{KIE1}. 
\begin{lemma}\label{lovasz}
Suppose that $G^{\prec}$ is a 2-colorable online graph and $\AA$ uses at least $k$ colors on a connected component $C$ of $G^{\prec}_i$. Let $v\in C$ be a vertex colored $k$ by algorithm $\AA$. Then both partite sets of $C$ contain at least $2^{k/2-1}$ vertices with distance at most $2^{k/2}$ from $v$.
\end{lemma}

\begin{proof}
We argue by induction on $k$ and note that the base step is trivial. For the induction step observe that if $\AA$ assigns color $k+2$ to $v_i$ then $\AA$ must have already assigned color $k$ to some vertex $v_p\in I_2$ and color $k+1$ to some other vertex in $I_2$. Thus $\AA$ must have assigned color $k$ to some vertex $v_q\in I_1$. $\AA$ assigned the same color to $v_p$ and $v_q$, hence $v_p$ and $v_q$ must be in separate components of $G^{\prec}_t$ where $t=\max\{p,q\}$. 
Thus by the induction hypothesis each of the color classes of these connected components must have at least $2^{ k/2-1}$ vertices and so the color classes of the components of $v_i$ have at least  $2^{(k+2)/2-1}$ vertices.

If $I_2 = N_{2^{k/2}}(v_i)$ the lemma holds. Otherwise, there exists $u\in I_2$ with $\mathrm{dist}(u,v_i)>2^{k/2}$, therefore there is a path in $C$ with length at least $2^{k/2}$ containing $v_i$, so the assertion follows.
\end{proof}

Now we modify Algorithm~$\B_n$ for graphs with high oddgirth.
\begin{quote}
\textbf{Algorithm $\BO_{n,d}$} \\
Let $d$ be a positive integer.
Consider the input sequence $v_1{\prec}\ldots{\prec}v_n$ of an online graph $G^{\prec}$ with oddgirth at least $4d+1$.
Set $r=(n/(d\log_2 d))^{1/2}$.
Initialize by setting $H_i=\emptyset$ for all $i\in [r]$ and $W_i=\emptyset$ for all $r<i<2r\log_2 d + 2n/(rd)$. At the $s$-th stage the algorithm processes the vertex $v_s$ as follows.
\begin{enumerate}
\item If there exists $i\in [r]$ such that the subgraph induced by $H_i \cup \{v_s\}$ is 2-colorable 
and algorithm $\AA$ uses at most $2\log_2 d$ colors to color $H_i \cup \{v_s\}$,
then let $j$ be the least such $i$. Set $H_j = H_j \cup \{v_s\}$ and color $v_s$ (in subgraph induced by $H_j$) by algorithm $\AA$. ($\AA$ uses colors $2(j-1)\log_2 d + 1, \ldots, 2j\log_2 d$ to color $H_j$.)
\item Otherwise, if there exists $i>r$ such that $v_s\in N_{d,\mathrm{odd}}(W_i)$ then color $v_s$ with the least such $i$.
\item Otherwise, let $j$ be the least integer $i>r$ such that $W_i=\emptyset$. Set $W_j= N^{\prec}_{d,\mathrm{odd}}(v_s)\cap\left(\bigcup_{\ell=1}^{r}H_{\ell}\right)$ and color $v_s$ with $j$. 
\end{enumerate}
\end{quote}

\begin{theorem}
Suppose that $d>1$ integer.
Algorithm $\BO_{n,d}$ produces a coloring of any graph $G^{\prec}$ on $n$ vertices having oddgirth at least $4d+1$ with at most $4(n\log_2 d / d)^{1/2}$ colors.
\end{theorem}

\begin{proof}
First we prove that $\BO_{n,d}$ produces a coloring.
Assume to the contrary that two adjacent vertices $x{\prec}y$ have the same color $j$. Clearly $y$ is not colored by Step~1. Thus $j>r$ and hence $x$ is not colored by Step~1. 
Since only the first vertex colored $j$ can be colored by Step~3, $y$ must be colored by Step~2. If $x$ is colored by Step~3 then $W_j\subset N_{d,\mathrm{odd}}(x)$ and $y\in N_{d,\mathrm{odd}}(W_j)$ so there exists $z\in W_j$ that both $\mathrm{dist}(x,z)\le d$ and $\mathrm{dist}(y,z)\le d$ are odd. But then $G^{\prec}$ contains a closed walk with odd length at most ${2d+1}$ so it contains an odd cycle with length at most ${2d+1}$, a contradiction. 
If $x$ is colored by Step~2 then both $x$ and $y$ are in $N_{d,\mathrm{odd}}(W_j)$ so there exist (not necessarily distinct) $x',y'\in W_j$ with odd distances $\mathrm{dist}(x',x)\le d$, $\mathrm{dist}(y',y)\le d$, $\mathrm{dist}(x',z)\le d$ and $\mathrm{dist}(y',z)\le d$ where $z$ is the first vertex colored with $j$. In this case $G^{\prec}$ contains a closed walk with odd length at most $4d+1$, a contradiction. So $\BO_{n,d}$ produces a coloring.

Now we give an upper bound for the number of colors used by $\BO_{n,d}$. At most $2r\log_2 2d$ colors are used in Step~1.
Let $j>r$ and $z_j$ be the first vertex colored $j$.
Now we show that $|W_j\cap H_k| \ge d$ for all $k\le r$. Since $z_j$ is not colored by Step~1 we have two cases.
In the first case the subgraph induced by $H_k\cup\{z_j\}$ contains an odd cycle containing $z_j$ with length at least $4d+1$ so $|W_j\cap H_k| \ge d$ by the definition of $W_j$.
In the second case the subgraph induced by $H_k\cup\{z_j\}$ has is 2-colorable but algorithm $\AA$ uses at least $2\log_2 d +1$ colors to color this subgraph. From the proof of Lemma~\ref{lovasz} it follows that $|W_j\cap H_k| \ge 2^{(2\log_2 d)/2 - 1} = d/2$. Since $H_k\cap H_{\ell}=\emptyset$ if $k\ne\ell$, we get that $|W_j|\ge rd/2$.
Since $W_i\cap W_j = \emptyset$ if $i\ne j$, at most $n/(rd)$ colors are used in Steps~2 and 3. Thus 
\begin{eqnarray*}
\chi_{GB}(G^{\prec}) &\le& 2r\log_2 d + \frac{n}{rd/2}\\
&=& 2\left(\frac{n}{d\log_2 d}\right)^{1/2}\log_2 d + \frac{2n}{d\left(\frac{n}{d\log_2 d}\right)^{1/2}}\\
&=& 4\left(\frac{n\log_2 d}{d}\right)^{1/2}.
\end{eqnarray*}
\end{proof}

We note that running $\BO_{n,d}$ on $G$ with oddgirth $4d+1$ choosing $r=(n/(2d))^{1/2}$ and exploiting that the subgraphs induced by $H_j$ $(j\in [r]$) are 2-colorable we get that $\chi(G)\le 2r + \frac{n}{rd} = 2\left(\frac{2n}{d}\right)^{1/2}$.
Denley \cite{DEN} proved that if $G$ is a graph on $n$ vertices with oddgirth at least $2k+3 \ (k\ge 2)$ then its independence number 
$\alpha(G)\ge \left(\frac{n}{k}\right)^{\frac{k}{k+1}}(\log n)^{\frac{1}{k+1}}$.
From this one can prove a stronger upper bound for the chromatic number of graphs with large oddgirth than the above.

Note that all algorithms contained in the paper can be easily modified to color arbitraty input graphs.

\section*{Acknowledgement} The author wish to thank P\'eter Hajnal for his guidance and suggestions and Mih\'aly Hujter for his helpful observations. This research has been supported by OTKA Grant K76099.

\end{document}